\documentclass{amsart}
\usepackage{graphicx} % Required for inserting images
\usepackage{amsthm, amsmath, amssymb, tikz, bm}
\usepackage{mathtools}
\usepackage{xifthen}
\usepackage{comment}
\usepackage[margin=4.5cm]{geometry}
\usepackage[shortlabels]{enumitem}
\usepackage[colorlinks=true,
linkcolor=blue,citecolor=blue,
urlcolor=blue]{hyperref}

\usetikzlibrary{calc,shapes, backgrounds}
\usetikzlibrary[patterns]

\newtheorem{theorem}{Theorem}[section]
\newtheorem{lemma}[theorem]{Lemma}
\newtheorem{sublemma}{}[theorem]
\newtheorem{corollary}[theorem]{Corollary}

\newcommand{\romannum}[1]{\romannumeral#1\relax}

\makeatletter
\newdimen\p@renwd \setbox0=\hbox{\phantom B} \p@renwd=\wd0
\def\bordersquare#1{\begingroup \m@th
  \setbox0=\vbox{\def\cr{\crcr\noalign{\kern2pt\global\let\cr=\endline}}
      \ialign{$##$\hfil\kern2pt\kern\p@renwd&\thinspace\hfil$##$\hfil
        &&\quad\hfil$##$\hfil\crcr
        \omit\strut\hfil\crcr\noalign{\kern-\baselineskip}
        #1\crcr\omit\strut\cr}}
  \setbox2=\vbox{\unvcopy0 \global\setbox1=\lastbox}
  \setbox2=\hbox{\unhbox1 \unskip \global\setbox1=\lastbox}
  \setbox2=\hbox{$\kern\wd1\kern-\p@renwd \left[ \kern-\wd1
    \global\setbox1=\vbox{\box1\kern2pt}
    \vcenter{\kern-\ht1 \unvbox0 \kern-\baselineskip} \,\right]$}
  \null\;\vbox{\kern\ht1\box2}\endgroup}
\makeatother

\newcommand{\cP}{\mathcal{P}}
\newcommand{\cQ}{\mathcal{Q}}
\newcommand{\cR}{\mathcal{R}}

\newcommand{\cT}{\mathcal{T}}

\newcommand{\GF}{\operatorname{GF}}

\newcommand{\W}{\mathcal{W}}

\makeatletter
\newcommand*{\rom}[1]{\expandafter\@slowromancap\romannumeral #1@}

\makeatother

\title[A pair of elements in no non-spanning circuits]{The connected binary matroids with a pair of elements in no non-spanning circuits}
\author{Wayne Ge}
\address{Mathematics Department \\
Louisiana State University \\ Baton Rouge, LA}
\email{yge4@lsu.edu}

\author{James Oxley}
\address{Mathematics Department \\
Louisiana State University \\ Baton Rouge, LA}
\email{oxley@math.lsu.edu}

\author{Jagdeep Singh}
\address{Department of Mathematics and Statistics \\
Mississippi State University \\ Mississippi State, MS}
\email{singhjagdeep070@gmail.com}

\date{\today}

\begin{document}

\begin{abstract}
Let $M$ be a simple connected binary matroid, and let $e$ and $f$ be distinct elements of $M$. It is well known that, when the only circuits containing $e$ are spanning, $M$ is a circuit with at least three elements. This paper proves that if every circuit containing $\{e,f\}$ is spanning, then the canonical tree decomposition of $M$ is a path in which each vertex is labeled by a circuit, a copy of $U_{1,3}$, or a binary spike having one non-tip element deleted.
\end{abstract}
\maketitle

\section{Introduction}\label{sec: intro}

The terminology in this paper follows~\cite{Oxl11}. The following result is well known. A proof may be found in~\cite[Lemma~3.1(\romannum{1})]{oxsin}.

\begin{lemma}\label{lem: one element in no nonspanning circuits}
    Let $e$ be an element of a simple connected binary matroid $M$ with at least two elements. If every circuit containing $e$ is spanning, then $M$ is a circuit.
\end{lemma}

Our main result determines precisely when some pair of elements in a connected binary matroid $M$ is contained in no non-spanning circuits. The statement of this theorem relies on Cunningham and Edmonds's~\cite{CE1980} canonical tree decomposition of a connected matroid, the details of which are given in Section~\ref{sec: canonical tree decompositions}. For $r\geq 3$, we denote by $Z_r$ the rank-$r$ binary spike with tip $t$ and ground set $\{x_1,y_1,x_2,y_2,\dots,x_r,y_r,t\}$ where $\{x_i,y_i,t\}$ is a triangle for each $i\in\{1,2,\dots,r\}$ (see~\cite[p.~473]{Oxl11}).

\begin{theorem}\label{thm: main (short)}
    Let $M$ be a simple connected binary matroid, and let $e$ and $f$ be distinct elements of $M$. Suppose that every circuit containing $\{e,f\}$ is spanning. Then the canonical tree decomposition of $M$ is a path $\cP$ every vertex of which is labeled by a circuit of size at least three, by $U_{1,3}$, or by $Z_r\backslash y_r$ for some $r\geq 3$, where the ends of $\cP$ are not labeled by $U_{1,3}$.
\end{theorem}

A more detailed statement of this result appears in Section~\ref{sec: proof of the main thm} along with a proof of the result. If we drop the requirement that $M$ is simple from Theorem~\ref{thm: main (short)}, the only change to the result is that we may add arbitrarily many elements in parallel to each element of the matroid described there.

In addition to its intrinsic interest, Theorem~\ref{thm: main (short)} is motivated by Ge's work~\cite{unavoidable_flats_regular} on unavoidable flats of large connected regular matroids. The main result of that paper is the following.

\begin{theorem}\label{thm:large graphic flat}
    For every positive integer $k$, there is an integer $f(k)$ such that every connected regular matroid of rank at least $f(k)$ has a connected graphic flat of rank at least $k$.
\end{theorem}

A key ingredient in the proof of this theorem is the specialization of Theorem~\ref{thm: main (short)} to regular matroids. The statement of this corollary uses a special class of graphs called clean ladders. These were introduced by Allred, Ding, and Oporowski~\cite{SDO2020}, who identified such graphs as one of the families of unavoidable induced subgraphs for large $2$-connected simple graphs. Let $B_3$ denote the graph that consists of two vertices joined by three edges. A graph $L$ is a \emph{clean ladder} if it can be constructed via a path decomposition whose vertices are labeled, in order, by $G_0,G_1,\dots, G_n$ where
\begin{enumerate}
    \item[(\romannum{1})] each $G_i$ is isomorphic to $K_4$, $K_3$, or $B_3$;
    \item[(\romannum{2})] each edge of the path is the basepoint of the $2$-sum between its endpoints;
    \item[(\romannum{3})] neither $G_0$ nor $G_n$ is isomorphic to $B_3$;
    \item[(\romannum{4})] no two consecutive vertices of the path are both isomorphic to $B_3$; and
    \item[(\romannum{5})] whenever $0<i<n$ and $G_i\cong K_4$, the two edges of $G_i$ that are basepoints of $2$-sums form a matching.
\end{enumerate}

\begin{corollary}\label{cor: regular matroids}
    Let $M$ be a simple connected regular matroid, and let $e$ and $f$ be distinct elements of $M$. Suppose that every circuit of $M$ containing both $e$ and $f$ is spanning. Then $M\cong M(L)$ for some clean ladder $L$.
\end{corollary}

This has the following immediate consequence for graphs.

\begin{corollary}
    Let $G$ be a simple $2$-connected graph, and let $e$ and $f$ be distinct edges of $G$. If every cycle of $G$ containing both $e$ and $f$ is Hamiltonian, then $G$ is a clean ladder.
\end{corollary}

Ge and Oxley~\cite{cc-minors} showed that the cycle matroids of clean ladders are one of the families of unavoidable flats for large simple connected cographic matroids. Ge~\cite{unavoidable_flats_regular} used Corollary~\ref{cor: regular matroids} to prove that the cycle matroids of clean ladders are one of the families of unavoidable flats of large simple connected regular matroids.

In Section~\ref{sec: canonical tree decompositions}, we describe Cunningham and Edmonds's~\cite{CE1980} tree decomposition of a connected matroid. Section~\ref{sec: the 3-con case} proves Theorem~\ref{thm: main (short)} when $M$ is $3$-connected. Finally, in Section~\ref{sec: proof of the main thm}, we prove a strengthened version of Theorem~\ref{thm: main (short)}, characterizing all simple connected binary matroids having a pair $\{e,f\}$ of elements that are contained in no non-spanning circuits.

\section{Canonical tree decompositions}\label{sec: canonical tree decompositions}
Theorem~\ref{thm: main (short)} uses the canonical tree decomposition of a connected matroid. In this section, we describe this decomposition. It is well known that a connected matroid $M$ that is not $3$-connected can be written as a $2$-sum $M_1\oplus_2 M_2$ of two connected matroids. If $M_1$ or $M_2$ is not $3$-connected, it too can be decomposed as a $2$-sum. Repeating this process, we can eventually decompose $M$ into $3$-connected matroids, circuits, and cocircuits. Cunningham and Edmonds~\cite{CE1980} recorded this decomposition via a tree, each vertex of which is labeled by a matroid. A detailed description of this decomposition may be found in~\cite[pp.~307--310]{Oxl11}. Formally, a \emph{tree decomposition} of a connected matroid $M$ is a tree $\cT$ with vertex set $\{M_1,M_2,\dots,M_k\}$ where each $M_i$ is a connected matroid and $E(\cT)=\{e_1,e_2,\dots,e_{k-1}\}$. Moreover,
\begin{enumerate}
    \item[(\romannum{1})] $E(M)=(E(M_1)\cup E(M_2)\cup\dots\cup E(M_k))-\{e_1,e_2,\dots,e_{k-1}\}$;
    \item[(\romannum{2})] if $M_{j_1}$ and $M_{j_2}$ are joined by an edge $e_i$ then $E(M_{j_1})\cap E(M_{j_2})=\{e_i\}$; otherwise $E(M_{j_1})\cap E(M_{j_2})=\emptyset$;
    \item[(\romannum{3})] $|E(M_i)|\geq 3$ for all $i$ unless $|E(M)|<3$, in which case, $k=1$ and $M_1=M$; and
    \item[(\romannum{4})] $M$ can be obtained from $\cT$ by repeatedly taking an edge $e_i$ with ends $M_{j_1}$ and $M_{j_2}$, contracting $e_i$ from the current tree and labeling the resulting composite vertex by $M_{j_1}\oplus_2 M_{j_2}$. In particular, $M$ corresponds to the single-vertex tree $\cT/e_1,e_2,\dots,e_{k-1}$. 
\end{enumerate}
Cunningham and Edmonds~\cite{CE1980} proved that every connected matroid $M$ has a tree decomposition in which every vertex is labeled by a circuit, a cocircuit, or a $3$-connected matroid with at least four elements, and no two adjacent vertices are both labeled by circuits or are both labeled by cocircuits. This tree, which is unique up to its edge labels, is called the \emph{canonical tree decomposition} of $M$.

\section{The 3-connected case}\label{sec: the 3-con case}

In this section, we prove Theorem~\ref{thm: main (short)} when $M$ is $3$-connected. The next lemma is an immediate consequence of a result of Oxley and Reid~\cite[Theorem~1.4(\romannum{1})]{Oxley_Reid}.

\begin{lemma}\label{lem: roundness of 4-wheel}
    Let $M$ be a $3$-connected binary matroid with an $M(\W_4)$-minor. If $e$ and $f$ are distinct elements of $M$, then $M$ has a minor $N$ such that $N\cong M(\W_4)$ and $\{e,f\}\subseteq E(N)$.
\end{lemma}

Recall from the introduction that $Z_r$ denotes the binary $r$-spike with tip $t$ where $y_r$ is some element other than $t$. The following theorem of Oxley~\cite[Theorem~(2.1)]{Oxley_no_W4} characterizes the $3$-connected binary matroids with no $M(\W_4)$-minor; see also~\cite[Theorem~12.2.21]{Oxl11}.

\begin{theorem}\label{thm: no W4-minor}
    Let $M$ be a binary matroid. Then $M$ is $3$-connected with no $M(\W_4)$-minor if and only if one of the following holds.
    \begin{enumerate}
        \item[(\romannum{1})] $M$ is isomorphic to one of $Z_r$, $Z_r^*$, $Z_r\backslash y_r$, and $Z_r\backslash t$ for some $r\geq 3$; or
        \item[(\romannum{2})] $M$ is isomorphic to one of $U_{0,1}$, $U_{1,1}$, $U_{1,2}$, $U_{1,3}$, and $U_{2,3}$.
    \end{enumerate}
\end{theorem}

The next lemma says that the hypothesis that every circuit containing both $e$ and $f$ is spanning is inherited by every minor that still contains $e$ and $f$.

\begin{lemma}\label{lem: preserve over taking minors}
    Let $e$ and $f$ be distinct elements of a matroid $M$ such that every circuit of $M$ containing both $e$ and $f$ is spanning. Let $N$ be a minor of $M$ with $\{e,f\}\subseteq E(N)$. Then every circuit of $N$ containing both $e$ and $f$ is spanning in $N$.
\end{lemma}

\begin{proof}
    Write $N=M\backslash X/Y$, where $X$ and $Y$ are disjoint subsets of $E(M)-\{e,f\}$. Let $C$ be a circuit of $N$ containing both $e$ and $f$. Then there is a circuit $C'$ of $M\backslash X$ such that
    \[
        C\subseteq C'\subseteq C\cup Y.
    \]
    Since circuits of $M\backslash X$ are circuits of $M$, the set $C'$ is a circuit of $M$ containing both $e$ and $f$. Hence $C'$ is spanning in $M$. As $C'\subseteq E(M)-X$, this implies that $C'$ is spanning in $M\backslash X$. Since $C'\subseteq C\cup Y$, the set $C\cup Y$ is spanning in $M\backslash X$. Therefore
    \[
        r_N(C)
        = r_{M\backslash X}(C\cup Y)-r_{M\backslash X}(Y)
        = r(M\backslash X)-r_{M\backslash X}(Y)
        = r(N).
    \]
    Thus $C$ is spanning in $N$.
\end{proof}

The matroid $Z_r\backslash y_r$ is represented over $\GF(2)$ by the matrix $A$ in Figure~\ref{fig: matrix A}.

\begin{center}
    \begin{figure}[htb]
    \hbox to \hsize{
	\hfil
    $
    \bordersquare
    {
          &x_1&x_2&x_3&\cdots&x_{r-1}&x_r&y_1&y_2&y_3&\cdots&y_{r-1}&t\cr
          &1&0&0&\cdots&0&0&0&1&1&\cdots&1&1\cr
          &0&1&0&\cdots&0&0&1&0&1&\cdots&1&1\cr
          &0&0&1&\cdots&0&0&1&1&0&\cdots&1&1\cr
          &\vdots&\vdots&\vdots&\ddots&\vdots&\vdots&\vdots&\vdots&\vdots&\ddots&\vdots&\vdots\cr
          &0&0&0&\cdots&1&0&1&1&1&\cdots&0&1\cr
          &0&0&0&\cdots&0&1&1&1&1&\cdots&1&1\cr
    }
    $
	\hfil
    }
    \caption{A matrix $A$ that represents $Z_r \backslash y_r$ over $\GF(2)$.}\label{fig: matrix A}
    \end{figure}
    \end{center}
Since $A$ has the form $\left[\, I_r \mid D \,\right]$ where $D$ is symmetric, we deduce that $Z_r\backslash y_r$ is self-dual. Moreover, in $(Z_r\backslash y_r)^*$, the element $x_r$ is the tip. Sometimes, $Z_r\backslash y_r$ is called a rank-$r$ binary spike with tip $t$ and \emph{cotip} $x_r$.

We now characterize the $3$-connected binary matroids with two prescribed elements in no non-spanning circuits. We note that this occurs precisely when $M$ is a rank-$r$ binary spike with a tip and a cotip and that the pair $\{e,f\}$ consists of the tip and the cotip.

\begin{theorem}\label{thm: the 3-con case}
    Let $M$ be a $3$-connected binary matroid with at least four elements. Let $e$ and $f$ be distinct elements of $M$, and suppose that every circuit containing both $e$ and $f$ is spanning. Then $M\cong Z_r\backslash y_r$ for some $r\geq 3$. Moreover, there is an isomorphism from $M$ to $Z_r\backslash y_r$ under which $\{e,f\}$ is mapped to $\{t,x_r\}$.
\end{theorem}

\begin{proof}
    Suppose first that $M$ has an $M(\W_4)$-minor. By Lemma~\ref{lem: roundness of 4-wheel}, $M$ has a minor $N$ such that $N\cong M(\W_4)$ and $\{e,f\}\subseteq E(N)$. It is straightforward to check that $N$ has a non-spanning circuit containing both $e$ and $f$, contradicting Lemma~\ref{lem: preserve over taking minors}. Thus $M$ has no $M(\W_4)$-minor.

    Since $M$ has at least four elements, by Theorem~\ref{thm: no W4-minor}, $M$ is isomorphic to one of
    \[
        Z_r,\quad Z_r^*,\quad Z_r\backslash y_r,\quad Z_r\backslash t
    \]
    for some $r\geq 3$. Observe that, in the Fano matroid $F_7$, every pair of elements is contained in a triangle, and, in $F_7^*$, every pair of elements is contained in a $4$-element circuit. These circuits are non-spanning. Hence, by Lemma~\ref{lem: preserve over taking minors}, $M$ has no $F_7$-minor or $F_7^*$-minor containing both $e$ and $f$.

    Suppose that $M\cong Z_r$ or $M\cong Z_r^*$ for some $r\geq 3$. Since $Z_3\cong F_7$ and, for each $r\geq 4$ and each pair $\{x,y\}\subseteq E(Z_r)$, there is a $Z_3$-minor of $Z_r$ containing $\{x,y\}$, it follows that $M$ has an $F_7$-minor or an $F_7^*$-minor containing both $e$ and $f$, a contradiction.

    We may now assume that $M\cong Z_r\backslash y_r$ or $M\cong Z_r\backslash t$ for some $r\geq 3$. Suppose that $M\cong Z_r\backslash t$. If $r\geq 4$, then every pair of elements is contained in a $4$-element circuit, which is non-spanning, a contradiction. If $r=3$, then
    \[
        Z_3\backslash t\cong Z_3\backslash y_3.
    \]
    Thus $M\cong Z_r\backslash y_r$ for some $r\geq 3$.

    It remains to identify the pair $\{e,f\}$ in $Z_r\backslash y_r$. If $r=3$, then $Z_3\backslash y_3\cong M(K_4)$. Since $\{e,f\}$ is not contained in a triangle of $Z_3\backslash y_3$, the corresponding pair of edges in $K_4$ forms a matching. It is straightforward to check that there is an isomorphism from $M$ to $Z_3\backslash y_3$ under which $\{e,f\}$ is mapped to $\{t,x_3\}$.

    Now suppose that $r\geq 4$. If $x_r\notin\{e,f\}$, then $\{e,f\}$ is contained in either a triangle or a $4$-element circuit, a contradiction. Thus we may assume that $e=x_r$. If $f\neq t$, then $\{e,f\}$ is contained in a non-spanning $r$-element circuit, again a contradiction. Hence $\{e,f\}=\{x_r,t\}$.
\end{proof}

\section{The proof of the main result}\label{sec: proof of the main thm}

In this section, we prove the following strengthening of Theorem~\ref{thm: main (short)}. We also prove Corollary~\ref{cor: regular matroids}.

\begin{theorem}\label{thm: main (detailed)}
    Let $M$ be a simple connected binary matroid, and let $e$ and $f$ be distinct elements of $M$. Then every circuit of $M$ containing $\{e,f\}$ is spanning if and only if the canonical tree decomposition of $M$ is a path $\cP$ with vertices, in order, $P_0,P_1,\dots,P_n$ where
    \begin{enumerate}
        \item[(\romannum{1})] each vertex is labeled by a circuit of size at least three, by $U_{1,3}$, or by a copy of $Z_r\backslash y_r$ for some $r\geq 3$;
        \item[(\romannum{2})] neither $P_0$ nor $P_n$ is labeled by $U_{1,3}$;
        \item[(\romannum{3})] if $n=0$ and $P_0$ is a binary spike with a tip and a cotip, this pair of elements is $\{e,f\}$;
        \item[(\romannum{4})] if $n\geq 1$, then one of $e$ and $f$ is in $E(P_0)$ and the other is in $E(P_n)$, and, when $P_i$ is labeled by a spike with a tip and a cotip,
        \begin{enumerate}
            \item[(a)] for $0<i<n$, the tip and the cotip are the two basepoints of $P_i$, and
            \item[(b)] for $i\in\{0,n\}$, the pair consisting of the tip and the cotip coincides with the pair consisting of the single basepoint in $P_i$ and the single member of $\{e,f\}$ in $P_i$.
        \end{enumerate}
    \end{enumerate}
\end{theorem}

\begin{proof}
We first show that if the only circuits of $M$ containing $\{e,f\}$ are spanning, then $M$ is as described. Let $\cT$ be the canonical tree decomposition of $M$. Let $\cP$ be the minimal subtree of $\cT$ containing the vertices that contain $e$ and $f$. First, we prove that

\begin{sublemma}\label{sublem: T=P}
    $\cT=\cP$.
\end{sublemma}

Suppose that $\cT - V(\cP)$ has a component $\cR$. Let $\cQ=\cT - V(\cR)$, and let $p$ be the edge of $\cT$ that meets both $\cR$ and $\cQ$. Let $M(\cR)$ and $M(\cQ)$ be the matroids corresponding to $\cR$ and $\cQ$. Then
\[
    M = M(\cR) \oplus_2 M(\cQ)
\]
where $p$ is the basepoint of this $2$-sum. Observe that $\{e,f\}\subseteq E(M(\cQ))$.

If $r(M(\cR))=1$, then, since $\cT$ is a canonical tree decomposition, $\cR$ is labeled by a cocircuit. Hence $E(M(\cR))-\{p\}$ contains at least two elements, and any two such elements form a circuit of $M(\cR)$ avoiding $p$. This circuit is also a circuit of $M$, so $M$ is not simple, a contradiction. Thus $r(M(\cR))\geq 2$.

Suppose $M(\cQ)$ has a circuit $C$ containing both $e$ and $f$ but not $p$. Then $C$ is also a circuit of $M$. Moreover,
\[
    r_M(C) \leq r(M(\cQ)) < r(M),
\]
since
\[
    r(M)=r(M(\cQ))+r(M(\cR))-1
\]
and $r(M(\cR))\geq 2$. This contradiction shows that every circuit of $M(\cQ)$ containing both $e$ and $f$ also contains $p$.

Let $D$ be a circuit of $M(\cR)$ containing $p$. If $D$ is not spanning in $M(\cR)$, then, for every circuit $C$ of $M(\cQ)$ containing both $e$ and $f$, the set
\[
    (C\cup D)-p
\]
is a non-spanning circuit of $M$ containing $\{e,f\}$, a contradiction. Thus every circuit of $M(\cR)$ containing $p$ is spanning.

We now show that $M(\cR)$ is simple. Since $M(\cR)$ is connected and $M$ is simple, if $\{x,y\}$ is a two-element circuit of $M(\cR)$, then $p\in\{x,y\}$, say $\{x,y\}=\{p,x\}$. Let $C$ be a circuit of $M(\cQ)$ containing $\{e,f\}$. Then $(C-\{p\})\cup\{x\}$ is a circuit of $M$ that is non-spanning, a contradiction. Hence $M(\cR)$ is simple.

We now know that $M(\cR)$ is a simple connected binary matroid with at least two elements. By Lemma~\ref{lem: one element in no nonspanning circuits}, $M(\cR)$ is a circuit. Hence $\cR$ consists of a single vertex labeled by a circuit.

Let $N$ be the vertex of $\cQ$ adjacent to $\cR$. Since $\cT$ is a canonical tree decomposition, $N$ is not labeled by a circuit. Thus $N$ is labeled either by a cocircuit or by a $3$-connected matroid. In both cases, $N\backslash p$ is connected, and hence $M(\cQ)\backslash p$ is connected. Therefore $M(\cQ)\backslash p$ has a circuit $C$ containing both $e$ and $f$. This circuit is also a circuit of $M$, and it is non-spanning, a contradiction. We conclude that~\ref{sublem: T=P} holds.\\

Since $\cP$ is the minimal subtree of $\cT$ containing the vertices that contain $e$ and $f$, it follows that $\cP$ is a path with vertices
\[
    P_0,P_1,\ldots,P_n,
\]
in order, where $e\in E(P_0)$ and $f\in E(P_n)$. If $n=0$, the result follows immediately from Theorem~\ref{thm: the 3-con case}. Hence we may assume that $n\geq 1$. For each $i\in\{1,2,\dots,n\}$, let $p_i$ be the edge of $\cP$ joining $P_{i-1}$ and $P_i$. Then the circuits of $M$ that contain $\{e,f\}$ are precisely the sets of the form
\[
    C_0\triangle C_1\triangle\dots\triangle C_n,
\]
where $C_0$ is a circuit of $P_0$ containing $\{e,p_1\}$, and $C_n$ is a circuit of $P_n$ containing $\{p_n,f\}$, while $C_i$ is a circuit of $P_i$ containing $\{p_i,p_{i+1}\}$ for each $i\in\{1,2,\dots,n-1\}$.

If, for some $i\in\{0,1,\dots,n\}$, the circuit $C_i$ is non-spanning in $P_i$, then
\[
    C_0\triangle C_1\triangle\dots\triangle C_n
\]
is non-spanning in $M$, a contradiction. Thus no $P_i$ has a non-spanning circuit containing its two distinguished elements, namely $\{e,p_1\}$ in $P_0$, the pair $\{p_i,p_{i+1}\}$ in $P_i$ for $1\leq i\leq n-1$, and $\{p_n,f\}$ in $P_n$.

Because $M$ is simple and $\cT$ is the canonical tree decomposition of $M$, Theorem~\ref{thm: the 3-con case} implies that every $P_i$ is labeled by a circuit, by $U_{1,3}$, or by a matroid isomorphic to $Z_r\backslash y_r$ for some $r\geq 3$. The identification of the distinguished elements follows immediately from Theorem~\ref{thm: the 3-con case}. We conclude that if every circuit of $M$ containing $\{e,f\}$ is spanning, then $M$ is as described in the theorem. We omit the routine proof that if $M$ is as described in the theorem, then every circuit containing $\{e,f\}$ is spanning.
\end{proof}

We conclude by obtaining Corollary~\ref{cor: regular matroids} as a consequence of Theorem~\ref{thm: main (detailed)}.

\begin{proof}[Proof of Corollary~\ref{cor: regular matroids}]
Let $\cT$ be the canonical tree decomposition of $M$. Suppose that $\cT$ has a vertex labeled by $N$, where
\[
    N\cong Z_r\backslash y_r
\]
for some $r\geq 4$. Since vertex labels in the canonical tree decomposition are minors of $M$, the matroid $M$ has an $N$-minor. But $Z_r\backslash y_r$ has an $F_7$-minor whenever $r\geq 4$. Hence $r=3$. Since
\[
    Z_3\backslash y_3 \cong M(K_4),
\]
and every isomorphism from $Z_3\backslash y_3$ to $M(K_4)$ maps the distinguished pair $\{t,x_3\}$ to a pair of non-adjacent edges of $K_4$, the path decomposition given by Theorem~\ref{thm: main (detailed)} is precisely the decomposition of a clean ladder. Hence $M$ is the cycle matroid of a clean ladder.
\end{proof}

\end{document}